\documentclass[12pt]{article}
\RequirePackage{amsmath,amsfonts,amssymb,amsthm}
\usepackage{graphicx,psfrag,epsf}
\usepackage{enumerate}
\usepackage{natbib}
\usepackage[dvipsnames]{xcolor}
\usepackage[colorlinks,citecolor=OliveGreen,linkcolor=blue,backref=true]{hyperref}
\usepackage{booktabs}
\usepackage{dsfont}
\usepackage{here}
\usepackage{rotating,multirow}
\usepackage{mathtools}
\usepackage{enumitem}

\usepackage{nameref,zref-xr} % nameref before zref-xr
\zxrsetup{toltxlabel} % allows to use the \ref command as usual
\newcommand{\trace}[1]{\operatorname{tr}\left(#1\right)}
\newcommand{\norm}[1]{\left\lVert#1\right\rVert}

\DeclarePairedDelimiter{\abs}{\lvert}{\rvert}

\newcommand{\eps}{\varepsilon}
\newcommand{\sign}[1]{\text{sign}\left\{#1\right\}}

\newcommand{\beps}{\boldsymbol{\varepsilon}}

\newcommand{\bEta}{\boldsymbol{\eta}}
\newcommand{\bBeta}{\boldsymbol{\beta}}

\newcommand{\bx}{\mathbf{x}}
\newcommand{\by}{\mathbf{y}}
\newcommand{\bR}{\mathbf{R}}
\newcommand{\hbR}{\hat{\mathbf{R}}}
\newcommand{\bC}{\mathbf{C}}
\newcommand{\hbC}{\hat{\mathbf{C}}}
\newcommand{\bRho}{\boldsymbol{\rho}}

\DeclareMathOperator*{\aic}{AIC}
\DeclareMathOperator*{\bic}{BIC}
\DeclareMathOperator{\mspe}{MSPE}
\DeclareMathOperator{\bmspe}{\textbf{MSPE}}
\DeclareMathOperator{\mric}{MRIC}
\DeclareMathOperator{\vmric}{VMRIC}
\DeclareMathOperator{\mi}{MI}
\DeclareMathOperator{\bmi}{\textbf{MI}}
\DeclareMathOperator{\vi}{VI}
\DeclareMathOperator{\bvi}{\textbf{VI}}
\DeclareMathOperator{\E}{E}

\newtheorem{theorem}{Theorem}[]
\newtheorem{remark}{Remark}[]
\newtheorem{assumptions}{Assumptions}[]
\newtheorem{definition}{Definition}[]
\newtheorem{proposition}{Proposition}[]
\newtheorem{lemma}{Lemma}[]

\newcommand{\blind}{0}

\begin{document}

\def\spacingset#1{\renewcommand{\baselinestretch}%
{#1}\small\normalsize} \spacingset{1}

%%%%%%%%%%%%%%%%%%%%%%%%%%%%%%%%%%%%%%%%%%%%%%%%%%%%%%%%%%%%%%%%%%%%%%%%%%%%%%

\if0\blind
{
  \title{\bf A multivariate extension of the Misspecification-Resistant Information Criterion.}
  \author{Gery Andrés Díaz Rubio\\
    Department of Statistical Sciences, University of Bologna, Italy\\
    and \\
    Simone Giannerini \\
    Department of Statistical Sciences, University of Bologna, Italy\\
    and \\
    Greta Goracci \\
    Faculty of Economics and Management, Free University of Bolzano-Bozen, Italy
    }
  \maketitle
} \fi

\if1\blind
{
  \bigskip
  \bigskip
  \bigskip
  \begin{center}
    {\LARGE\bf  A multivariate extension of the Misspecification-Resistant Information Criterion.}
\end{center}
  \medskip
} \fi
%\vspace*{-1cm}
%\bigskip
\begin{abstract}
The Misspecification-Resistant Information Criterion (MRIC) proposed in [H.-L. Hsu, C.-K. Ing, H. Tong: \emph{On model selection from a finite family of possibly misspecified time series models}. The Annals of Statistics. 47 (2), 1061--1087 (2019)] is a model selection criterion for univariate parametric time series that enjoys both the property of consistency and asymptotic efficiency. In this article we extend the MRIC to the case where the response is a multivariate time series and the predictor is univariate. The extension requires novel derivations based upon random matrix theory. We obtain an asymptotic expression for the mean squared prediction error matrix, the vectorial MRIC and prove the consistency of its method-of-moments estimator. Moreover, we prove its asymptotic efficiency. Finally, we show with an example that, in presence of misspecification, the vectorial MRIC identifies the best predictive model whereas traditional information criteria like AIC or BIC fail to achieve the task.
\end{abstract}

\noindent%
\vfill

\newpage
\spacingset{1.45} % DON'T change the spacing!

%% **************************************************************************************
\section{Introduction}\label{sec1}
The model selection step is a fundamental task in statistical modelling and its implementation typically depends upon the objective of the exercise. In the time series framework the focus is on either forecasting future values or describing/controlling the process that has generated the data (DGP). A good model selection criterion must feature a good ability to identify the model with the ``best'' fit to future values, in a specified sense. In particular, in the parametric time series framework, we can identify two main properties. The first one is consistency, i.e., the ability to select the true DGP with probability one as the sample size diverges. This assumes that a true model exists and is among the set of candidate models. If either the set of candidate models does not contain the true DGP, or, for some reason, a true model cannot be postulated, then a selection criterion should be asymptotically efficient, for instance, in the mean square sense, i.e. it minimizes the mean squared prediction error as the sample size diverges. Starting from the seminal work of Akaike, \cite{AKA1973} a plethora of model selection criteria has been proposed. These include Akaike's AIC \cite{AKA1973, AKA1974}, Schwarz's Bayesian Information Criterion (BIC) \cite{SCH1978}, and Rissanen's Minimum Description Length (MDL) \cite{RIS1978}. Such criteria paved the way for various extensions dealing with different unsolved issues. For instance, the AIC is efficient but not consistent (i.e. it leads to select overfitting models), whereas the BIC is consistent but not efficient, see \cite{HSU2019} for a discussion.
\par
A recent development for model selection in possibly misspecified parametric time series models in the fixed-dimensionality setting is given by the Misspecification-Resistant Information Criterion (hereafter $\mric$) \cite{HSU2019}. Fixed-dimensionality means that the number of observations increases to infinity while the number of ‘true’ parameters is finite. In this respect, the MRIC provides a solution to the original research question of Akaike: it enjoys both consistency, in case the true model is included as a candidate, and asymptotic efficiency when a true model either cannot be assumed or is not included. Moreover, when the number of variables in the model grows with the sample size, the $\mric$ can achieve asymptotic efficiency, without the need for additional criteria. Finally, in the high-dimensional setting, the $\mric$ can be used together with appropriate model selection criteria to identify the best predictive models. The $\mric$ is based upon the additive decomposition of the mean squared prediction error in a term that depends upon the misspecification level and a term that measures the sampling variability of the predictor. The idea is to select the model with smallest variability among those that minimize the misspecification index.
\par
The appealing properties of the $\mric$ make it an ideal tool for omnibus time series model selection but, to date, only the univariate response case has been studied \cite{HSU2019}. In this work we extend the $\mric$ to multivariate time series with a single regressor as to obtain the vectorial MRIC (hereafter $\vmric$). As it will be clear, such an extension does not easily derive from the univariate case since it requires dealing with the dependence structure within the components of the vector of forecasting error and hence relies upon random matrix theory. Such multivariate extension can be used in all those models where many time series depend upon a single regressor, like for instance, in econometrics, where many interest rates depend upon a single macroeconomic indicator, such as inflation. Other possible applications include dimension reduction and hedging, which is intimately connected to the problem of model selection \cite{BES16}.
\par
The rest of the paper is organized as follows: in Section~\ref{sec:notation} we introduce the notation and in Section~\ref{subsec:MRIC} summarize the available results for the univariate case; in Section~\ref{sec:VMIRC} we extend the MRIC approach to multivariate time series with a single regressor. In particular, in Section~\ref{subsec:MSPE_dec} we obtain the asymptotic decomposition of the Mean Squared Prediction Error (hereafter MSPE) matrix  into two parts: the first one is linked to the goodness of fit of the model and the second one depends upon the prediction variance. In Section~\ref{subsec:consistency} we present the $\vmric$ and derive a consistent estimator for it, whereas in Section~\ref{subsec:eff}, we prove the asymptotic efficiency of the $\vmric$. Section~\ref{sec:example} presents an example to assess the effect of misspecification in the $\vmric$ framework. All the proofs are detailed in Section~\ref{sec:proof}. \ref{appendx} contains auxiliary technical lemmas.
\section{Notation and preliminaries}\label{sec:notation}
For each $t$, let $\{\bx_t\}$ and $\{\by_t\}$, with $\bx_t=(x_{t,1},\dots,x_{t,m})^\top$ and $\by_t=(y_{t,1},\dots,y_{t,w})^\top$, be two weakly
%strictly \plum{Gery: Perhcè strictly e non wealky?}
stationary stochastic processes defined over the probability space $\left(\Omega, \mathcal{F}, \mathbb{P}\right)$. When $m=1$ ($w=1$, respectively) we write $x_t$ ($y_t$). Given a vector $\mathbf{v}$ and a matrix $\mathbf{M}$, we use $\norm{\mathbf{v}}$ and $\norm{\mathbf{M}}$ to refer to the $\mathcal{L}_2$ vectorial norm  and the matrix norm induced by the Euclidean norm, respectively. We write $o(1)$ ($o_p(1)$) to indicate a sequence that converges (in probability) to zero and $O(1)$ ($O_p(1)$) to indicate a sequence that is bounded (in probability). Moreover, let $\{c_n\}$ be a sequence of scalar random variables whereas $\{\mathbf{v}_n\}$ and $\{\mathbf{M}_n\}$ are sequences of random vectors and random matrices, respectively. We adopt the following notation: $\mathbf{v}_n =o_p(c_n)$ if $\norm{\mathbf{v}_n}/c_n=o_p(1)$; $\mathbf{v}_n=O_p(c_n)$,  if $\norm{\mathbf{v}_n}/c_n=O_p(1)$,
$\mathbf{M}_n =o_p(c_n)$ if $\norm{\mathbf{M}_n}/c_n=o_p(1)$; $\mathbf{M}_n = O_p(c_n)$ if $\norm{\mathbf{M}_n}/c_n=O_p(1)$.
For further details on matrix algebra see \cite{SEB2008, HOR2013, ODE2018}, for multivariate time series see \cite{REI1993, LUT2005, TSA2013}, and for asymptotic tools for vector and matrices, see \cite{JIA2010}.%
\par
Let $\{(\bx_t,\by_t),t \in \{1,\dots,n\}\}$ be the observed sample, and divide the interval $[1,n]$ into the \emph{training set} $[1,N]$ and the \emph{test set} $[N+1,N+h]$, with $h$ being the forecasting horizon. Note that $\bx_t$ can contain both endogenous and exogenous variables, therefore, Model~(\ref{eqn:for_mod}) encompasses many different models including, inter alia, VAR and VARX models. We denote $\bar{\bx}=n^{-1}\sum_{t=1}^{n}\bx_t$ and $\bar{\by}=n^{-1}\sum_{t=1}^{n}\by_t$, i.e. the two sample means. Without loss of generality assume $\E[\bx_t]=\E[\by_t]=\boldsymbol 0$. In order to forecast $\by_{n+h}$, $h\geq 1$, we adopt the following $h$-step ahead forecasting Model:
\begin{equation}\label{eqn:for_mod}
	\by_{t+h}= \mathbf{B}_h \bx_t+\beps_t^{(h)},
\end{equation}
\par\noindent
where $\mathbf{B}_h$ is a $(w\times m)$ matrix and $\beps_t^{(h)}$ is the vector containing the $w$ $h$-step ahead forecast errors; as before, if $w=1$ we write $\eps_t^{(h)}$.
\begin{remark}
Since the model can possibly be misspecified, the prediction error vector $\beps_t^{(h)}$ can be serially correlated, and also correlated with $\bx_s$, $s\neq t$. Moreover, the multivariate framework differs from  \cite{HSU2019} in different key aspects. For instance, $(i)$ the components of the error vector can be cross-correlated, and $(ii)$ $\bx_t \beps_t^{(h)}$ and $\bx_k \beps_k^{(h)}$, for $t\neq k$, can also be both serially and cross correlated.
\end{remark}
Define
\begin{equation}\label{eqn:def_R}
	\hbR=N^{-1} \sum_{t=1}^{N}\bx_t\bx_t^\top\qquad\text{ and }\qquad \bR =\E[\bx_1 \bx_1^\top].
\end{equation}
Then, the ordinary least squares estimator (hereafter OLS) of $\mathbf{B}_h$ results:
\begin{equation}\label{eqn:ols}
	\hat{\mathbf{B}}_n(h)= \hbR^{-1}\left(N^{-1}\sum_{t=1}^{N} \bx_t \by_{t+h}^\top\right).
\end{equation}
When $w=1$, $\mathbf{R}$ and $\mathbf{B}$ become $R$ and $\boldsymbol\beta$, respectively. The prediction of $\by_{n+h}$, $h\geq 1$, is given by
\begin{equation}\label{eqn:for_y}
	\hat{\by}_{n+h}=\hat{\mathbf{B}}_n(h)\bx_n
\end{equation}
and the corresponding Mean Squared Prediction Error matrix is
\begin{equation}\label{eqn:MSPE}
	\bmspe_h=\E\left[(\by_{n+h}-\hat{\by}_{n+h})(\by_{n+h}-\hat{\by}_{n+h})^\top\right].
\end{equation}
\subsection{The MRIC for parametric univariate time series models}\label{subsec:MRIC}
In \cite{HSU2019}, the authors focused on the case $w=1$ and $m\geq1$. Under appropriate conditions, they obtained the following asymptotic decomposition of $\mspe$:
\begin{align}
	\mspe_h &= \E\left[(y_{n+h}-\hat{y}_{n+h})^2\right] = \mi_h+n^{-1}(\vi_h+o(1)),\label{eqn:MSPE_dec}\\
\text{with}\quad\mi_h &= \E\left[\left(\eps_n^{(h)}\right)^2\right],\quad
\vi_h = \trace{\mathbf{R}^{-1}\mathbf{C}_{h,0}} +2\sum_{s=1}^{h-1}\trace{\mathbf{R}^{-1}\mathbf{C}_{h,s}},\nonumber
\end{align}
\noindent
 where $\mathbf{C}_{h,s} = \E\left[\bx_1\bx_{1+s}^\top \eps_1^{(h)} \eps_{1+s}^{(h)}\right]$, $s\geq0$, is the cross-covariance matrix between the regressors and the $h$-step ahead prediction error at lag $s$.
\begin{remark}
The first part of Eq.~(\ref{eqn:MSPE_dec}) is the Misspecification Index ($\mi$), linked to the goodness-of-fit of the model and coincides with the $h$-step ahead prediction error variance. The second component is the Variability Index ($\vi$), which depends upon the variance of the $h$-step ahead predictor, $\hat{y}_{n+h}= \hat{\boldsymbol{\beta}}_n^\top (h) \mathbf{x}_n$, and is also linked to the bias of the estimator of $\boldsymbol{\beta}_h$.
\end{remark}
\noindent
Based upon the above decomposition, the $\mric$ is defined as follows:
\begin{align}
	\mric_h &= \hat{\mi}_h+\frac{\alpha_n}{n}\hat{\vi}_h,\label{eqn:MRIC_uni}
\end{align}
with $\hat{\mi}_h$ and $\hat{\vi}_h$ being the estimators of $\mi_h$ and $\vi_h$ respectively, i.e.:
\[
\hat{\mi}_h = N^{-1}\sum_{t=1}^{N}\left(\hat{\eps}_t^{(h)}\right)^2,\quad
\hat{\vi}_h = \trace{\hat R^{-1}\hat{\mathbf{C}}_{h,0}} + 2 \sum_{s=1}^{h-1} \trace{\hat R^{-1} \hat{\mathbf{C}}_{h,s}},
\]
\noindent
where $\hat{\mathbf{C}}_{h,s} = (N-s)^{-1}  \sum_{t=1}^{N-s}\bx_t\bx_{t+s}^\top \hat{\eps}_t^{(h)}\hat{\eps}_{t+s}^{(h)}$ and
$\hat{\eps}_t^{(h)}=y_{t+h}-\hat{\bBeta}_n(h)\bx_t$ is the estimated forecast error; $\alpha_n$ is a penalization term sequence such that, as $n$ increases:
\begin{align}\label{eqn:penalty}
\frac{\alpha_n}{\sqrt{n}}\to+\infty\qquad\text{ and }\qquad \frac{\alpha_n}{n}\to0.
\end{align}
\noindent
It is shown that  $\hat{\mi}_h$ and $\hat{\vi}_h$ are consistent estimators of ${\mi}_h$ and ${\vi}_h$, moreover the asymptotic efficiency of the $\mric$ is proved. By minimizing this criterion, the model which minimizes $\vi$ among those with minimum $\mi$ is selected. Among other features, the $\mric$ is particularly helpful in situations where competing models present the same goodness-of-fit and the same number of parameters.
\begin{remark}
The type of penalty considered in \cite{HSU2019} is similar to that used in \cite[p. 230]{SHI1989} for the correctly specified case.
\end{remark}
\section{A multivariate extension of the MRIC framework}\label{sec:VMIRC}

In this section we extend the $\mric$ approach to the case where the response is a multivariate time series ($w\geq2$) and the predictor is univariate ($m=1$), for a generic $h$-step ahead forecast.
\noindent
Hence, Model~(\ref{eqn:for_mod}) reduces to $\by_{t+h}=\bBeta_h x_t+ \beps_t^{(h)}$, namely:
\begin{equation}\label{eqn:for_mod_case1}
	\begin{cases}
		y_{t+h,1}= \beta_{h,1}x_t+\eps_{t,1}^{(h)} \\
		y_{t+h,2}= \beta_{h,2}x_t+\eps_{t,2}^{(h)}\\
        \vdots\\
        y_{t+h,w}= \beta_{h,w}x_t+\eps_{t,w}^{(h)}
	\end{cases}
\end{equation}

\subsection{Asymptotic decomposition of the MSPE matrix}\label{subsec:MSPE_dec}
We extend the asymptotic representation of the $\mspe_h$ defined in (\ref{eqn:MSPE_dec}) which is the key step to derive the $\vmric$ in this multivariate framework. We rely upon the following assumptions, which are the natural multivariate extensions of those in \cite{HSU2019}.
\begin{assumptions}
\begin{align*}
	%\tag{C1}
	(\text{C}1)	& \quad \exists \ q_1 > 5, 0<K_1<\infty : \ \text{for any } 1 \leq n_1 < n_2 \leq n,\\
			    & \E\left[
							\left\lvert \left( n_2-n_1+1 \right)^{-1/2}
								\sum_{t=n_1}^{n_2} x_{t}^2 - \E\left[ x_{t}^2 \right]
							\right\rvert^{q_1}
							\right]\leq	 K_1.\\
	%\tag{C2}
	(\text{C}2)	&\quad 1. \ \mathbf{C}_{h,s} = \ \E\left[ \boldsymbol{\eps}_t^{(h)} x_t \left(\boldsymbol{\eps}_{t+s}^{(h)} x_{t+s}\right)^\top  \right] \perp t,\\
				&\quad 2. \ \E\left[ x_1 x_n \eps_{1,i}^{(h)} \eps_{n,j}^{(h)}  \right] = \ o(n^{-1}) \ \forall \ i,j \in\{ 1, \dots, w\}.\\
	%\tag{C3}
	(\text{C}3)	& \quad 1. \sup_{-\infty<t<\infty} \E\left[ | x_t |^{10} \right] < \infty,\\
				& \quad 2. \sup_{-\infty<t<\infty} \E\left[ \norm{\boldsymbol{\eps}_{t}^{(h)}}^{6} \right] < \infty.\\
	%\tag{C4}
	(\text{C}4)	& \quad \exists \ 0 < K_2 < \infty : \ \text{for} \ 1\leq n_1 < n_2 \leq n,
				 \quad \E\left[\norm{\left(n_2-n_1+1\right)^{-\frac{1}{2}} \sum_{t=n_1}^{n_2} \boldsymbol{\eps}_t^{(h)} x_t }^5 \right] < K_2.\\
	%\tag{C5}
	(\text{C}5)	&	 \quad \text{For any } q>0, \ \E\left[ \left|\hat{R}^{-1} \right|^q \right] = O(1).\\
	%\tag{C6}
	(\text{C}6)	& \quad \exists  \mathcal{F}_t \subseteq \mathcal{F}, \mathcal{F}_t \ \text{an increasing sequence of } \sigma\text{-fields such that: }\\
				&	\quad 1. \ %\mathbf{x}_t=
					x_t \ \text{is} \ \mathcal{F}_t\text{-measurable}\\
				&	\quad 2. \ \sup_{-\infty<t<\infty} \E\left[ \left|\E\left[ x_t^2 \mid  \mathcal{F}_{t-k}\right] - R\right|^3\right] = \ o(1), \text{ as } k \rightarrow\infty,\\
				&	\quad 3. \ \sup_{-\infty<t<\infty} \E\left[ \norm{\E\left[ \boldsymbol{\eps}_{t}^{(h)} x_t \mid \mathcal{F}_{t-k}\right]}^3\right] = \ o(1), \text{ as } k \rightarrow\infty.			
	\end{align*}
\end{assumptions}
\begin{theorem}\label{thm:MSPE_dec_case1}
	Under the regularity conditions (C1) -- (C6), the asymptotic expression of the $\bmspe_h$ defined in (\ref{eqn:MSPE}) results
	\begin{align}\label{eqn:MSPE_dec_case1}
		&N \left\{ \E\left[ \left( \by_{n+h} - \hat{\by}_{n+h} \right) \left( \by_{n+h} - \hat{\by}_{n+h} \right)^\top - \E\left[
		\beps_{n}^{(h)} \beps_{n}^{(h)^\top} \right] \right] \right\}\\
		& =  R^{-1} \E\left[   \left(\beps_{1}^{(h)} x_1\right)\left(\beps_{1}^{(h)} x_1\right)^\top \right]\nonumber\\
		&+ R^{-1} \E\left[\sum_{s=1}^{h-1} \left\{\left( \beps_1^{(h)} x_1\right) \left(\beps_{s+1}^{(h)} x_{s+1}\right)^\top + \left(\beps_{s+1}^{(h)} x_{s+1}\right) \left( \beps_1^{(h)} x_1 \right)^\top\right\} \right] + o(1).\nonumber
	\end{align}
\end{theorem}
\noindent
\subsection{VMRIC and its consistent estimation}\label{subsec:consistency}
In this section we introduce the $\vmric$. Let $\{\alpha_n\}$ be the penalization term sequence defined as in Eq.~(\ref{eqn:penalty}).
% In analogy to the univariate case of  Eq.~(\ref{eqn:MRIC_uni}) we  follow a conservative approach and define  the $\vmric_h$ as the sum of two norms:
\begin{align}
\vmric_h &= \norm{{\bmi}_h} + \norm{ \frac{\alpha_n}{n} {\bvi}_h}  \label{eqn:VMRIC}\\
\text{where } \bmi_h &= \E\left[\left(\beps_t^{(h)}\beps_t^{(h)\top}  \right)\right],\quad
\bvi_h = R^{-1} \left( \bC_{h,0} + \sum_{s=1}^{h-1} \left( \bC_{h,s} + \bC_{h,s}^\top \right) \right),\nonumber\\
\bC_{h,s} &= \E\left[\left(x_t\beps_t^{(h)}\right)\left(x_t\beps_t^{(h)}\right)^\top\right].\nonumber
\end{align}
The $\vmric$ can be estimated via the method of moments as to obtain:
\begin{align}
	\hat\vmric_h &\equiv \norm{\hat{\bmi}_h} + \norm{ \frac{\alpha_n}{n} \hat{\bvi}_h \label{eqn:VMRIC_hat}},\\
\text{where } \hat{\bmi}_h &= N^{-1}\sum_{t=1}^{N}\left(\hat{\beps}_t \hat{\beps}_t^\top \right),\quad
\hat{\bvi}_h = \hat{R}^{-1} \left[ \hbC_{h,0} + \sum_{s=1}^{h-1} \left( \hbC_{h,s} + \hbC_{h,s}^\top \right)  \right],\nonumber
\end{align}
and  $\hbC_{h,s} = (N-s)^{-1}  \sum_{t=1}^{N-s} x_t x_{t+s} \hat{\beps}_t \hat{\beps}_{t+s}^\top$, and $\hat{\beps}_t=\by_{t+h}-\hat{\bBeta}_n(h) x_t$ is the estimated forecast error vector.
\par
\noindent In Theorem~\ref{thm:MOME_case1} we prove that $\hat{\bmi}_h$ and $\hat{\bvi}_h$ are consistent estimators of ${\bmi}_h$ and ${\bvi}_h$, respectively. Theorem~\ref{thm:MOME_case1} relies upon  the following assumptions, that are less restrictive with respect to (C1) -- (C6). For further discussions on the assumptions see \cite[][Remark~1--3, p. 1073]{HSU2019}.
\begin{assumptions}
For each $0\leq s \leq h-1$, we assume the following:
	\begin{align*}
		%\tag{A1}
		%(\text{A}1)
	(\text{A}1)	& \quad n^{-1} \sum_{t=1}^{n} \left( \beps_t^{(h)} \beps_t^{(h)\top}\right) = \E\left[  \beps_1^{(h)} \beps_1^{(h)\top} \right] + O_p \left(n^{-1/2}\right)\\%, \text{and any } 1\leq i,j \leq m,\\
		%\tag{A2}
(\text{A}2)
		%(\text{A}2)
		&\quad n^{-1} \sum_{t=1}^{n}  \left(x_t \beps_t^{(h)}\right) \left(x_{t+s}\beps_{t+s}^{(h)}\right)^\top = \boldsymbol{C}_{h,s} + o_p(1),\\
		%\tag{A3}
(\text{A}3)
		%(\text{A}3)
		& \quad n ^{-1/2} \sum_{t=1}^{n} x_t \beps_t^{(h)} = O_p(1).\\
		%\tag{A4}
(\text{A}4)
		%(\text{A}4)
		& \quad n ^{-1} \sum_{t=1}^{n} x_t^{2} = R + o_p(1),\\
		%\tag{A5}
(\text{A}5)
		%(\text{A}5)
		& \quad \sup_{-\infty < t < \infty} \E \left[\left\|\beps_t^{(h)}\right\|^4\right] +
		\sup_{-\infty<t<\infty} \E\left[\abs {x_t}^4 \right] < \infty.			
	\end{align*}
\end{assumptions}

\begin{theorem}\label{thm:MOME_case1}
	If Assumptions (A1) -- (A5) hold, then for the case $w\geq2$, and $m=1$ we obtain:
	\begin{align*}
		\hat{\bmi}_h&= \bmi_h +\, O_p (n^{-1/2}),\\
		\hat{\bvi}_h&= \bvi_h +\, o_p (1).
	\end{align*}
\end{theorem}
\subsection{Asymptotic efficiency}\label{subsec:eff}
In this section we prove the asymptotic efficiency of the $\vmric$ in the fixed dimensionality framework. To this end, let $\mathcal{M}$ be the set of $K$ candidate models; each model is indicated either by $\ell$ or $\kappa$, $1\leq \ell,\kappa\leq K$. Define the subsets $M_1$ and $M_2$ as follows:
\begin{equation}
	M_1 = \left\{
	\kappa: 1 \leq \kappa \leq K,
	\norm{\bmi_h(\kappa)} = \min_{1\leq \ell \leq K}
	\norm{\bmi_h(\ell)}
	\right\}
\end{equation}
\begin{equation}
	M_2 = \left\{
	\kappa: \kappa \in M_1,
	\norm{\bvi_h(\kappa)} = \min_{\ell \in M_1}
	\norm{\bvi_h(\ell)}
	\right\}.
\end{equation}
In short, for a given forecast horizon $h$, $M_1$ contains the models with the minimum $\bmi_h$ whereas in $M_2$ we are minimizing $\bvi_h$ among the candidates models in $M_1$. The definition of efficiency used in our framework is the same as that of \cite{HSU2019}:
\begin{definition}\label{def:eff}
  Given a sample of size $n$, a model selection criterion is said to be asymptotically efficient if it selects the model $\hat{\ell}_h$ such that
  $$\lim_{n\to\infty}\Pr \left( \hat\ell_h \in M_2\right) = 1.$$
\end{definition}
\begin{remark}
Alternative definitions of asymptotic efficiency for model selection are available. For instance, in the framework of linear stationary processes, \cite{SHI1980} defines the Mean Efficiency when a criterion attains asymptotically a lower bound for the sum of squared prediction errors. Also, the notion of Approximate Efficiency is given in \cite{SHI1984}. In \cite{LI1987}, a criterion that depends upon the ratio between loss functions is introduced. This latter definition is similar to the Loss Efficiency proposed in \cite{SHA1997}.
\end{remark}
The $\vmric$ selects the model with the smallest variability index among those that achieve the best goodness of fit. Hence, the selected model $\hat\ell_h$  is such that:
\begin{equation}\label{eq:vmric}
	\vmric_h \left( \hat{\ell}_h\right) \equiv
	\min_{1 \leq \ell \leq K}
	\norm{ \hat{\boldsymbol{\mi}}_h (\ell) } +
	\min_{ \ell \in M_1}
	\norm{ \frac{C_n}{n} \hat{\boldsymbol{\vi}}_h (\ell)}.
\end{equation}
In the next Theorem we show that the $\vmric$ is an asymptotic efficient model selection criterion in the sense of Definition~\ref{def:eff}.
\begin{theorem}\label{thm:EFF_case1}
	Assume that for each $ 1 \leq \ell \leq K$, $0\leq s \leq h-1$, Theorem~\ref{thm:MOME_case1} holds and let $\hat\ell_h$ be the model selected by the $\vmric$. Then we have that:
	\begin{align*}
\label{eq:theo3}		%\tag{B3}
\lim_{n\rightarrow \infty} \Pr \left( \hat{\ell}_h \in M_2\right) &= 1,
	\end{align*}
namely, the $\vmric$ is asymptotically efficient in the sense of Definition~\ref{def:eff}.
\end{theorem}

%% *************************************************************************************
%% *************************************************************************************
\section{Example: a misspecified bivariate AR(2) model}\label{sec:example}
%% *************************************************************************************

The aim of this section is twofold. First, we assess the goodness of the theoretical derivations and the finite sample behaviour of the method of moments estimator for the $\vmric$. Second, we show that in presence of misspecification the $\vmric$ leads to selecting the best predictive model (i.e. is asymptotically efficient) whereas both the $\aic$ and the $\bic$ fail to do so. In order to achieve the goals we consider a bivariate AR(2) DGP and use two misspecified predictive models for it: in Model 1 there is one omitted lagged predictor, whereas Model 2 uses only one non-informative predictor. We derive theoretically the Mean Square Prediction Error matrix and the $\vmric$ for both models and these show that Model 1 is a better predictive model over Model 2. Based on this, we assess the ability of the $\vmric$, and of the multivariate versions of the $\aic$ and $\bic$ to select the best model (Model 1) in finite samples and for different parameterizations.
\par
We start by providing the definition of misspecification. Consider an increasing sequence of $\sigma$-fields, $\left\{ \mathcal{G}_t \right\}$ such that $\sigma\left( \bx_s, s\leq t \right) \subseteq  \mathcal{G}_t  \subseteq \mathcal{F} $, where $\left\{ \bx_t \right\}$ is an $m$-dimensional weakly stationary process defined over the probability space $\left(\Omega, \mathcal{F}, \mathbb{P} \right)$.

\begin{definition}\label{def:miss}
	The $h$-step ahead forecasting model:
	\begin{equation}
		\by_{t+h} = \bBeta_h^\top \bx_t + \beps_t^{(h)},
	\end{equation}
	is correctly specified with respect to an increasing sequence of $\sigma$-fields, $\left\{ \mathcal{G}_t \right\}$ if
	\begin{equation}
		E\left[\by_{t+h} \mid \mathcal{G}_t \right] = \bBeta_h^\top \bx_t \ \text{a.s.}, \ \forall \ -\infty < t < \infty.
	\end{equation}
	Otherwise, it is misspecified.
\end{definition}

\begin{remark}\label{re:miss}
	%A consequence of Definition \ref{def:miss} is that, if we have a  \ac{DGP}  of some form and a possibly misspecified forecasting model, we obtain:
	The presence of misspecification implies that: $E\left[ \beps_t^{(h)} x_t \right] = \boldsymbol{0}$, while it is possible to have
$E\left[ \beps_t^{(h)} x_s \right] \neq \boldsymbol{0}$, for $s\neq t$, i.e. we have null simultaneous correlation and non-null cross correlation between the forecasting error vector and the regressor.
\end{remark}
%\subsection{Case 1}
Consider the following DGP :
\begin{equation}\label{eq:DGP}
	\by_{t+1} = \boldsymbol{a} w_t + \beps_{t+1},
\end{equation}
where $\boldsymbol{a}\neq \boldsymbol{0}$, $\left\{ \beps_t \right\}$ is a sequence of independent and identically distributed (hereafter i.i.d.) bivariate random vectors with $E\left[ \beps_1 \right] = \boldsymbol{0} $, $E\left[ \beps_1 \beps_1^{\top} \right] >\boldsymbol{0} $ and $w_t$ is the following scalar AR(2) process:
\begin{equation}\label{eq:AR2}
	w_t = \phi_1 w_{t-1} + \phi_2 w_{t-2} + \delta_t,
\end{equation}
where $\phi_1 \phi_2 \neq 0$ ,  $\left\{ \delta_t \right\}$ a sequence of i.i.d. random variables independent of $\left\{ \beps_t \right\}$ such that
$$E\left[ \delta_1 \right] = 0\quad\text{and}\quad E\left[ \delta_1^2 \right] = 1 - \phi_2^2 - \left\{ \phi_1^2 \frac{1+\phi_2}{1-\phi_2} \right\}.$$
%with $E\left[ \delta_1 \right] = 0$, $E\left[ \delta_1^2 \right] = 1 - \phi_2^2 - \left\{ \phi_1^2 \frac{1+\phi_2}{1-\phi_2} \right\}$ and .
Hence, we obtain $E\left[ w_t^2 \right] \equiv \gamma_w(0) = 1$, where $\gamma_w(j) = E\left[ w_t w_{t+j} \right]$ is the $j$-th lag autocovariance of $w$. \par
We consider the correctly specified $2$-step ahead forecasting model:
\begin{align}\label{eqn:fcst_mod_corr}
    \by_{t+2} &= \boldsymbol{a} w_{t+1} + \beps_{t+1}, \text{ which leads to } \nonumber\\
	\by_{t+2} &= \boldsymbol{a} \phi_1 w_t + \boldsymbol{a} \phi_2 w_{t-1} + \beps_t^{*(2)},
\end{align}
where $\beps_t^{*(2)} = \beps_{t+2} + \boldsymbol{a} \delta_{t+1}$. It can be easily proved that $E\left[ \beps_t^{*(2)} w_{t-j} \right] = \boldsymbol{0}$ for $j \geq 0$.
\par
Now, consider the following \textit{misspecified} model, Model 1:
\begin{align*}
	\by_{t+2}& = \bBeta w_t + \beps_t^{(2)},	\quad \text{with}\quad\bBeta = \frac{E\left[ \by_{t+2} w_t \right]}{V\left[ w_T \right]} = \boldsymbol{a} \left( \phi_1 + \frac{\phi_1 \phi_2}{1- \phi_2} \right).
\end{align*}
%where $\bBeta \equiv \frac{E\left[ \by_{t+2} w_t \right]}{V\left[ w_T \right]} = \boldsymbol{a} \left( \phi_1 + \frac{\phi_1 \phi_2}{1- \phi_2} \right)$.
The forecasting error results:
\begin{equation}\label{eqn:fcst_error_miss}
	\beps_t^{(2)} = \beps_t^{*(2)} - \boldsymbol{a} \phi_2 \left[ \frac{\phi_1}{1-\phi_2} w_t - w_{t-1} \right].
\end{equation}
\par\noindent
\begin{remark}
  In presence of misspecification $E\left[ \beps_{t}^{(2)} w_{t} \right] = \boldsymbol{0}$, whereas $E\left[ \beps_{t}^{(2)} w_{t-j} \right] \neq \boldsymbol{0}$ for $j\neq0$. We show that this occurs in our case:
\begin{align}
E\left[ \beps_{t}^{(2)} w_{t-j} \right]& = -\boldsymbol{a} \frac{\phi_2}{1-\phi_2} \left\{ \phi_1 E\left[ w_t w_{t-j} \right]  - \left( 1- \phi_2 \right) E\left[ w_{t-1} w_{t-j} \right]  \right\}\nonumber\\
& = -\boldsymbol{a} \frac{\phi_2}{1-\phi_2} \left\{\gamma_w(j+1) - \gamma_w (j-1)\right\},\nonumber
\end{align}
which is zero if $j=0$, otherwise this is generally not the case.
\end{remark}
We compute the theoretical value of the $\vmric$ by using Eq.~(\ref{eqn:VMRIC}).
After some routine algebra, we get:
\begin{equation}\label{eqn:example_mi2}
	\bmi = E\left[
	\beps_{n}^{(2)}
	\beps_{n}^{(2)^\top}
	\right] = \boldsymbol{\sigma}^2_\eps +
	\boldsymbol{a} \boldsymbol{a}^\top
	\left[
	\sigma^2_{\delta}
	+ \phi_2^2
	\left(
	1 - \gamma_w^2 (1)
	\right)
	\right],
\end{equation}

which highlights how the variance-covariance matrix of the $2$-step ahead forecast vector is equal to the {DGP}'s variance-covariance plus a bias term that depends upon the misspecification considered.
\par
Now we focus on the variability index $\bvi$. We get
\begin{align}\label{eqn:example_c20}
	 \bC_{2,0} = \boldsymbol{\sigma}^2_{\beps} +
	\boldsymbol{a} \boldsymbol{a}^\top \left\{
	\sigma^2_{\delta}
	+  \phi^2_2 \left(
	\gamma_w(1)^2 E\left[w_t^4\right]- 2 \gamma_w(1)E\left[w_t^3 w_{t-1}\right] +
	 E\left[w_t^2 w_{t-1}^2\right]
	\right)
	\right\}
\end{align}
and
\begin{align}\label{eqn:example_c21}
\bC_{2,1} =
\boldsymbol{a} \boldsymbol{a}^\top \gamma_w(1)\left(
b_1	E\left[w_{t-1}^3 w_{t-2}\right]  +
b_2	E\left[w_{t-1} w_{t-2}^3 \right] +
b_3 E\left[w_{t-1}^2 w_{t-2}^2\right]
\right),
\end{align}
where
\begin{align*}
	b_1= 2\phi_1\phi_2\gamma_w(1) - \phi_2,\qquad
	b_2= -\phi_2^2, \qquad
	b_3= \phi_2 \left(\phi_2\gamma_w(1) - 2\phi_1 + \gamma_w(1)^{-1} \right).
\end{align*}

Following Eq.~(\ref{eqn:VMRIC}), the results from Eq.~(\ref{eqn:example_mi2}), (\ref{eqn:example_c20}), and (\ref{eqn:example_c21}), deliver the $\vmric$ for this case.
\par
Now we consider a second misspecified model, Model 2:
\begin{equation}
\by_{t+2} = \bRho z_t + \bEta_t^{(2)},
\end{equation}
where $z_t$ is a weakly stationary linear AR(1) process independent of $w_t$:
\begin{equation}\label{eq:AR1}
z_t = \psi_1 z_{t-1} + \upsilon_t
\end{equation}
with $\psi_1 \in (-1, 1)$, and $\left\{ \upsilon_t \right\}$ is a sequence of i.i.d. random variables independent of both the error terms  $\left\{ \delta_t\right\}$ and $\left\{\beps_{t}\right\}$ such that $E[\upsilon_t] = 0$ and $E[\upsilon_t^2] = 1 - \psi_1^2$, delivering $E[z_t] = 0$ and $E[z_t^2]=1$. Thus, $z_t$ is uncorrelated with both $w_t$ and $\by_t$, therefore $\bRho=\boldsymbol{0}$. The forecasting error in this case results $\bEta_t^{(2)} = \boldsymbol{a} w_{t+1} + \beps_{t+2}$. Following similar arguments as above we obtain ${\bmi}$ and ${\bvi}$ for Model 2:
\begin{align}
{\bmi} &= \boldsymbol{\sigma}^2_{\beps} + \boldsymbol{a} \boldsymbol{a}^\top  \\
{\bvi} &= \boldsymbol{\sigma}^2_{\beps} + \boldsymbol{a} \boldsymbol{a}^\top (1 + 2 \psi_1 \gamma_w(1))
\end{align}
\par
As mentioned above, Model 1 is misspecified since it omits the lagged predictor $w_{t-1}$, while Model 2 only includes the non-informative predictor $z_t$.

% ***********************************************************************************
\subsection{Finite sample performance}
% ***********************************************************************************

First, we compare the above theoretical derivations with their sample counterpart. We consider three different parameterizations, presented in Table~\ref{tab:1}. Also, $\alpha_n = n^\alpha$ with $\alpha=0.85$. Note that, in order for Eq.~(\ref{eqn:penalty}) to hold, $\alpha$ must range in  $(0, 1)$. Further experiments showed that results are fairly robust if reasonable values of $\alpha$ are selected. For an empirical method to determine it, see \cite[Section 5]{hsu2019_supp}. We take the following variance/covariance matrix for the innovations:
$$E[\beps_t\beps_t^{\top}] =
\left[\begin{array}{cc}
	1 & 0.5 \\
	0.5 & 1
\end{array}\right].$$
\par
We compute both the {$\vmric$} for Model 1 and Model 2, and estimate the $\hat{\vmric}$ and $\hat{\vmric}$ on a large sample of $n=10^6$ observations. The results are shown in Table~\ref{tab:2} for the two models, where the theoretical $\vmric$ (rows 1 and 3) is compared with the estimated one (rows 2 and 4). The results seem to confirm the consistency of the estimator shown in Eq.~(\ref{eqn:VMRIC_hat}). Clearly, the $\vmric$ of Model 1 is consistently smaller than that of Model 2 and indicates its superior predictive capability.
%%%%%
\begin{table}[]
\centering	\caption{Parameters' combinations for the DGP of Eq.~(\ref{eq:DGP}),  (\ref{eq:AR2}), and (\ref{eq:AR1}).}\label{tab:1}
	\begin{tabular}{crrrrr}
		Case & $\phi_1$ & $\phi_2$ & $a_1$ & $a_2$ & $\psi_1$ \\
         \cmidrule(lr){2-6}
		1 &  0.4  & -0.75 &  1.50 & -2.00 &  0.80   			   \\
		2 & -0.4  & -0.45 & -0.75 & 1.25  & -0.65  			   \\
		3 &  0.3  & -0.80 &  1.00 & 0.50  & -0.75  			   \\
         \cmidrule(lr){2-6}
	\end{tabular}
\end{table}
%%%%%
\begin{table}[]
\centering	\caption{Theoretical and estimated {$\vmric$} of Models 1 and 2, for the three parameterizations of Table~\ref{tab:1}, computed on a data set of $n=10^6$ observations.}
	\label{tab:2}
	\begin{tabular}{ccccc}
		& \multicolumn{2}{c}{Model 1} & \multicolumn{2}{c}{Model 2} \\
        \cmidrule(lr){2-3}\cmidrule(lr){4-5}
Case & $\vmric$ & $\hat{\vmric}$ & $\vmric$ & $\hat{\vmric}$ \\
        \cmidrule(lr){2-3}\cmidrule(lr){4-5}
		1  & 6.671  & 6.636  & 7.914 & 7.902 \\
		2  & 2.777  & 2.768  & 3.164 & 3.168 \\
		3  & 2.801  & 2.784  & 2.994 & 2.993 \\
        \cmidrule(lr){2-3}\cmidrule(lr){4-5}
	\end{tabular}
\end{table}
%%%%%
The finite sample behaviour of the method of moments estimator of the $\vmric$ can be further appreciated in Table~\ref{tab:3} where we show their bias and Mean Squared Error (MSE).
%, computed as follows:
%\begin{align}
%\mathrm{Bias}_{} &= \norm{ E\left[\left(  \hat{\bmi} - {\bmi}\right)   + \frac{n^\alpha}{n} \left( \hat{\bvi} -{\bvi} \right)\right] }\\
%\mathrm{MSE}_{}  &= \norm{ E\left[
%	\left\{ \left(  \hat{\bmi} - {\bmi}\right)   + \frac{n^\alpha}{n} \left( \hat{\bvi} -{\bvi} \right) \right\}^2\right] }
%\end{align}
%\noindent
The results are based upon 1000 Monte Carlo replications and seem to indicate a rate of convergence of the order of $n^{-1}$.
\par
In Table~\ref{tab:4}, we show the percentages of correct model selection by the $\vmric$, compared with the multivariate version of the $\aic$ and $\bic$ for the three parameterizations of Table~\ref{tab:1}. For a sample size of $n=100$, both the $\aic$ and $\bic$ select the best predictive model in about 50\% of the cases and relying upon them is tantamount to tossing a fair coin. In such a case, the $\vmric$ selects the correct model in about $80\%$ of the cases and reaches $100\%$ for $n=1000$. On the contrary, for Case 3,  both the $\aic$ and $\bic$ cannot go above $64\%$ for a sample size as large as $n=10000$ observations and this is a general indication of their lack of asymptotic efficiency.

\begin{table}[]
\centering
	\caption{Bias and Mean-Squared Error (MSE) for the (method of moments) estimator of the $\vmric$ for the three parameterizations, $\alpha=0.85$ and different sample size $n$. The results are based upon $1000$ Monte Carlo replications.}
	\label{tab:3}
	\begin{tabular}{rcccccc}
 & \multicolumn{2}{c}{Case 1} & \multicolumn{2}{c}{Case 2} & \multicolumn{2}{c}{Case 3} \\
 \cmidrule(lr){2-3}\cmidrule(lr){4-5} \cmidrule(lr){6-7}
        $n$     & Bias  & MSE   & Bias  & MSE   & Bias  & MSE \\
 \cmidrule(lr){2-3}\cmidrule(lr){4-5} \cmidrule(lr){6-7}
		100   & 0.227 & 1.137 & 0.063 & 0.306 & 0.030 & 0.182 \\
		250   & 0.117 & 0.455 & 0.022 & 0.107 & 0.032 & 0.076 \\
		500   & 0.061 & 0.225 & 0.015 & 0.048 & 0.004 & 0.032 \\
		1000  & 0.019 & 0.109 & 0.010 & 0.023 & 0.002 & 0.015 \\
		2500  & 0.008 & 0.044 & 0.001 & 0.009 & 0.001 & 0.006 \\
		5000  & 0.009 & 0.023 & 0.001 & 0.004 & 0.003 & 0.003 \\
		10000 & 0.001 & 0.012 & 0.003 & 0.002 & 0.001 & 0.002 \\
		15000 & 0.004 & 0.008 & 0.001 & 0.001 & 0.002 & 0.001 \\
		30000 & 0.002 & 0.004 & 0.001 & 0.001 & 0.001 & 0.001 \\
  \cmidrule(lr){2-3}\cmidrule(lr){4-5} \cmidrule(lr){6-7}
	\end{tabular}
\end{table}
%%%%%
\begin{table}[]
\centering
\caption{Percentages of correctly selected models by the three information criteria for the three parameterizations and varying sample size $n$.}\label{tab:4}
\small			
\begin{tabular}{rccccccccc} %Case 7, Case 10, Case 2
&\multicolumn{3}{c}{Case 1} & \multicolumn{3}{c}{Case 2} & \multicolumn{3}{c}{Case 3} \\
\cmidrule(lr){2-4}\cmidrule(lr){5-7} \cmidrule(lr){8-10}	
$n$     & VMRIC & AIC & BIC & VMRIC & AIC & BIC &VMRIC & AIC & BIC \\
\cmidrule(lr){2-4}\cmidrule(lr){5-7} \cmidrule(lr){8-10}	
100   & 85.9 & 52.5 & 52.5 & 84.6 & 56.2 & 56.2 & 72.1 & 49.0 & 49.0  \\
1000  & 99.9 & 65.6 & 65.6 & 99.9 & 73.7 & 73.7 & 97.0 & 56.8 & 56.8 \\
10000 & 100  & 88.0 & 88.0 & 100  & 97.8 & 97.8 & 100  & 63.8 & 63.8 \\
\cmidrule(lr){2-4}\cmidrule(lr){5-7} \cmidrule(lr){8-10}	
\end{tabular}
\end{table}
%% **********************************************************************
\section{Proofs}\label{sec:proof}
%% **********************************************************************

In this section we detail the proofs of the three theorems. Hereafter all the derivations hold for any fixed $h\geq 1$; for the sake of presentation we write $ \beps_{t}$ instead of  $ \beps_{t}^{(h)}$.
 Remember that  $\{l_n\}$ indicates an increasing sequence  of positive integers such that:
\begin{equation}\label{eqn:ln}
  l_n \rightarrow \infty,\qquad \dfrac{l_n}{\sqrt{n}}=o\left(1\right)
\end{equation}
and define $a = n-l_n-h$ and $b=n-l_n-h+1$.

\subsection{Proof of Theorem~\ref{thm:MSPE_dec_case1}}\label{sec:A1}

The proof of Theorem~\ref{thm:MSPE_dec_case1} relies upon four propositions.

\begin{proposition}\label{prop:thm1_1}
	Under assumptions of Theorem~\ref{thm:MSPE_dec_case1}, it holds that:
	\begin{equation}
		N(\mathrm{I}) =   (\mathrm{III}) + o(1),\label{eqn:thm1_1}
	\end{equation}
	where
	\begin{align*}
		(\mathrm{I}) =- \E\left[x_n \hat{R}^{-1} \left( \boldsymbol{\hat{\Sigma}} \boldsymbol{\eps}_{n}^\top + \boldsymbol{\eps}_{n} \boldsymbol{\hat{\Sigma}}^\top \right) \right],\quad
		(\mathrm{III})= \ - \E\left[x_n R^{-1} \left( \boldsymbol{\hat{\Sigma}}_A \boldsymbol{\eps}_{n}^\top + \boldsymbol{\eps}_{n} \boldsymbol{\hat{\Sigma}}_A^\top\right) \right],
	\end{align*}
	with $\boldsymbol{\hat{\Sigma}} = \left( N^{-1} \sum_{t=1}^{N}x_t \boldsymbol{\eps}_{t}  \right)$ and  $\boldsymbol{\hat{\Sigma}}_A = \sum_{t=1}^{N}\boldsymbol{\varepsilon_{t}} x_t$.
\end{proposition}
\begin{proof}
Let $\boldsymbol{A}_1  = \sum_{t=1}^{N} \left(\beps_t x_t\right) \beps_n^\top$ and note that
\begin{align}
	\label{I-IIIdiv} \norm{(\mathrm{I})-(\mathrm{III}) }
	&= \norm{\E\left[x_n \left(\hat{R}^{-1}-R^{-1}\right)\left(\boldsymbol{A}_1+\boldsymbol{A}_1^\top\right)\right] }.
\end{align}
By using standard properties of the norm, (\ref{eqn:thm1_1}) follows upon proving that
\begin{equation}\label{eqn:thm1_1main}
 \norm{ \E\left[x_n \left(\hat{R}^{-1}-R^{-1}\right)\boldsymbol{A}_1^\top\right]}=o(1).
\end{equation}
Let
\begin{equation}\label{eqn:R_tilde}
\tilde{R} = \left(n-l_n\right)^{-1} \sum_{t=1}^{n-l_n} x_t^2.
\end{equation}
By adding and subtracting  $ \beps_n x_n \left(\tilde{R}^{-1} \left[\sum_{t=1}^{N}\left(\beps_t x_t\right)\right]^\top \right)$, we have
\begin{align}
 \E\left[x_n \left(\hat{R}^{-1}-R^{-1}\right)\boldsymbol{A}_1^\top\right]&=\E\left[\beps_n x_n \left(\hat{R}^{-1}-\tilde{R}^{-1}\right) \sum_{t=1}^{N} \beps_t^\top x_t\right]\nonumber\\
 &+ \E\left[\beps_n x_n \left(\tilde{R}^{-1}-R^{-1}\right) \sum_{t=1}^{N} \beps_t^\top x_t\right]\nonumber\\
 &= \E\left[\beps_n x_n \left(\hat{R}^{-1}-\tilde{R}^{-1}\right) \left(\sum_{t=1}^{N} \beps_t x_t\right)^\top\right]\label{eqn:1}\\
 &+ \E\left[\beps_n x_n \left(\tilde{R}^{-1}-R^{-1}\right) \left(\sum_{t=b}^{N} \beps_t x_t\right)^\top\right]\label{eqn:2}\\
 &+ \E\left[\beps_n x_n \left(\tilde{R}^{-1}-R^{-1}\right) \left(\sum_{t=1}^{a} \beps_t x_t\right)^\top\right]\label{eqn:3}.
\end{align}
We show below that the norms of (\ref{eqn:1}),  (\ref{eqn:2}) and (\ref{eqn:3}) are asymptotically negligible. Focus on the first one: by combining   conditions (C3), (C4), Lemma \ref{lma:technical}, and H\"older's inequality, it follows that $\norm{(\ref{eqn:1})}$ is bounded by
\begin{align*}
\E \left[\norm{\beps_n x_n \left(\hat{R}^{-1}-\tilde{R}^{-1}\right) \left(\sum_{t=1}^{N} \beps_t x_t\right)^\top}\right]
&\leq \E\left[\norm{\beps_n}^6\right]^{\frac{1}{6}} \E\left[\left\lvert x_n \right\rvert^{6} \right]^{\frac{1}{6}}  \E\left[\left\lvert \hat{R}^{-1}-\tilde{R}^{-1} \right\rvert^{3} \right]^{\frac{1}{3}}\\
& \times \E\left[\norm{ N^{\frac{1}{2}}N^{-\frac{1}{2}}\sum_{t=1}^{N} \beps_t x_t } ^{3}\right]^{\frac{1}{3}} = O\left(\frac{l_n}{n^{1/2}}\right),
\end{align*}
which converges to zero due to the definition of $l_n$ in (\ref{eqn:ln}). Similarly, we have that $\norm{(\ref{eqn:2})}$ is bounded by
\begin{align*}
	\E\left[\norm{\beps_n}^6\right]^\frac{1}{6}	\E\left[\left\lvert x_n\right\rvert^6\right]^\frac{1}{6} 	\E\left[\left\lvert\tilde{R}^{-1}-R^{-1}\right\rvert^3\right]^\frac{1}{3} \E\left[\norm{\left(\left(N-b+1\right)^{\frac{1}{2}}\left(N-b+1\right)^{-\frac{1}{2}}\sum_{t=b}^{N} \beps_t x_t\right)^\top}^3\right]^\frac{1}{3}.
\end{align*}
which is an $O\left(n^{-1/2}l_n\right)$ thereby vanishing asymptotically. Lastly, Condition (C6), Lemma \ref{lma:technical}, and H\"older's inequality imply that $\norm{(\ref{eqn:2})}$ is bounded by
\begin{align*}
 \E\left[\norm{\E\left[ \beps_t x_t \mid \mathcal{F}_{t-l_n}\right]}^3\right]^\frac{1}{3}  \E\left[\left\lvert\tilde{R}^{-1}-R^{-1}\right\rvert^{3}\right]^\frac{1}{3}  \E\left[\norm{a^{\frac{1}{2}} a^{-\frac{1}{2}}\sum_{t=1}^{a} \beps_t^\top x_t}^3\right]^\frac{1}{3}= o\left(1\right)
\end{align*}
and this completes the proof.
\end{proof}
\begin{proposition}\label{prop:thm1_2}
	Under assumptions of Theorem~\ref{thm:MSPE_dec_case1}, it holds that:
\begin{equation}\label{eqn:thm1_2}
		N(\mathrm{II}) =  (\mathrm{IV}) + o(1),
	\end{equation}
where
	\begin{align*}
		(\mathrm{II}) =  \E\left[\hat{R}^{-1} \boldsymbol{\hat{\Sigma}}x_n x_n \boldsymbol{\hat{\Sigma}}^\top \hat{R}^{-1}\right],\quad
		(\mathrm{IV}) =  \ \E\left[ \boldsymbol{\hat{\Sigma}}_B R^{-1} \boldsymbol{\hat{\Sigma}}_B^\top \right],
	\end{align*}
	with $\boldsymbol{\hat{\Sigma}}$ being defined in Proposition~\ref{prop:thm1_1} and $\boldsymbol{\hat{\Sigma}}_B = N^{-\frac{1}{2}} \sum_{t=1}^{N}\boldsymbol{\varepsilon_{t}} x_t.$
\end{proposition}
\begin{proof}
 Let $\boldsymbol{M}_1 = x_n \left(\hat{R}^{-1} - R^{-1}\right) \boldsymbol{\hat{\Sigma}}_B$ and $\boldsymbol{M}_2 = x_n R^{-1} \boldsymbol{\hat{\Sigma}}_B$. Since
\begin{align*}
	N \left(\mathrm{II}\right)= \E\left[\left(\boldsymbol{M}_1+ \boldsymbol{M}_2\right) \left(\boldsymbol{M}_1+ \boldsymbol{M}_2\right)^\top\right]
	&= \E\left[\boldsymbol{M}_1 \boldsymbol{M}_1^\top\right] + \E\left[\boldsymbol{M}_2\boldsymbol{M}_2^\top\right]\\
& + \E\left[\boldsymbol{M}_1 \boldsymbol{M}_2^\top\right] + \E\left[\boldsymbol{M}_2 \boldsymbol{M}_1^\top\right]
\end{align*}
the proof of (\ref{eqn:thm1_2}) reduces to show that the following conditions hold:
\begin{align}
\norm{\E\left[ \boldsymbol{M}_1 \boldsymbol{M}_1^\top \right]} &= o\left(1\right),\label{prop:cond1}\\
\norm{\E\left[ \boldsymbol{M}_1 \boldsymbol{M}_2^\top \right]} &= o\left(1\right),\label{prop:cond2} \\
\norm{\E\left[ \boldsymbol{M}_2 \boldsymbol{M}_2^\top \right]- (\mathrm{IV})}&= o\left(1\right).\label{prop:cond3}
\end{align}
Conditions (\ref{prop:cond1}) and (\ref{prop:cond2}) readily follow from Assumptions (C3) and (C4), Lemma \ref{lma:technical}, the non singularity of $R$ and H\"older's inequality:
\begin{align*}
	\E\left[\norm{\boldsymbol{M}_1 \boldsymbol{M}_1^\top }\right] &=  \E\left[\norm{x_n^2 \left(\hat{R}^{-1} - R^{-1}\right)^2 \boldsymbol{\hat{\Sigma}}_B \boldsymbol{\hat{\Sigma}}_B^\top }\right]
   \leq   \left(\E\left[\left\lvert	x_n	\right\rvert^{10}\right]\right)^\frac{1}{5}  \left(\E\left[\left\lvert\hat{R}^{-1}-R^{-1}\right\rvert^5\right]\right)^\frac{2}{5}\\  &\times\left(\E\left[\norm{\boldsymbol{\hat{\Sigma}}_B}^5\right]\right)^\frac{2}{5}= o\left(1\right);\\
   	\E\left[ \norm{\boldsymbol{M}_1\boldsymbol{M}_2^\top} \right] &= \E\left[\norm{x_n^2 \left(\hat{R}^{-1} - R^{-1}\right) R^{-1} \boldsymbol{\hat{\Sigma}}_B \boldsymbol{\hat{\Sigma}}_B^\top }\right]
	 \leq  \left(\E\left[\left\lvert x_n\right\rvert^{10}\right]\right)^\frac{1}{5}  \left(\E\left[\left\lvert\hat{R}^{-1}-R^{-1}\right\rvert 	^5\right]\right)^\frac{1}{5}\\
&\times  \left(\E\left[\left\lvert R^{-1}	\right\rvert^5\right]\right)^\frac{1}{5} \left(\E\left[\norm{\boldsymbol{\hat{\Sigma}}_B}^5\right]\right)^\frac{2}{5}= o\left(1\right).
\end{align*}
As concerns (\ref{prop:cond3}), decompose the vector $\boldsymbol{\hat{\Sigma}}_B$ as follows:
\begin{align*}
\boldsymbol{\hat{\Sigma}}_B = N^{-\frac{1}{2}} \sum_{t=1}^{N}\boldsymbol{\varepsilon_{t}} x_t = \boldsymbol{u} + \boldsymbol{w}, \quad
\text{ with } \quad \boldsymbol{u} = N^{-\frac{1}{2}} \sum_{t=1}^{a}\boldsymbol{\varepsilon_{t}} x_t \quad\text{and}\quad \boldsymbol{w} = N^{-\frac{1}{2}} \sum_{t=b}^{N}\boldsymbol{\varepsilon_{t}} x_t.
\end{align*}
Hence, we have that
\begin{align*}
\E\left[\boldsymbol{M}_2\boldsymbol{M}_2^\top\right]-  (\mathrm{IV}) &= \E\left[\boldsymbol{u} R^{-1} x_n x_n R^{-1} \boldsymbol{u}^\top\right] - \E\left[\boldsymbol{u} R^{-1} R R^{-1} \boldsymbol{u}^\top\right]\\
	& + \E\left[\boldsymbol{u} R^{-1} x_n x_n R^{-1} \boldsymbol{w}^\top\right] - \E\left[\boldsymbol{u} R^{-1} R R^{-1} \boldsymbol{w}^\top\right]\\
  & + \E\left[\boldsymbol{w} R^{-1} x_n x_n R^{-1} \boldsymbol{u}^\top\right] - \E\left[\boldsymbol{w} R^{-1} R R^{-1} \boldsymbol{u}^\top\right]\\
  & + \E\left[\boldsymbol{w} R^{-1} x_n x_n R^{-1} \boldsymbol{w}^\top\right] - \E\left[\boldsymbol{w} R^{-1} R R^{-1} \boldsymbol{w}^\top\right].
\end{align*}
The law of iterated expectations implies that:
\begin{align}
&\norm{\E\left[\boldsymbol{M}_2 \boldsymbol{M}_2^\top\right]-(\mathrm{IV})}\nonumber\\
&\leq \norm{\E\left[\boldsymbol{u} R^{-1} \left(\E\left[x_n^2 \mid \mathcal{F}_{n-l_n}\right]-R\right) R^{-1} \boldsymbol{u}^\top\right]}\label{norm1}\\
&+ \norm{\E\left[\boldsymbol{u} R^{-1} \left(\E\left[x_n^2 \mid \mathcal{F}_{n-l_n}\right]-R\right) R^{-1} \boldsymbol{w}^\top\right]}\label{norm2}\\
&+ \norm{\E\left[\boldsymbol{w} R^{-1} \left(\E\left[x_n^2 \mid \mathcal{F}_{n-l_n}\right]-R\right) R^{-1} \boldsymbol{u}^\top\right]}\label{norm3}\\
&+ \norm{\E\left[\boldsymbol{w} R^{-1} \left(\E\left[x_n^2 \mid \mathcal{F}_{n-l_n}\right]-R\right) R^{-1} \boldsymbol{w}^\top\right]}\label{norm4}.
\end{align}
By using arguments previously developed, it is easy to see that, under Assumptions (C4) and (C6), (\ref{norm1}) -- (\ref{norm4}) asymptotically vanish. Therefore, conditions (\ref{prop:cond1}) -- (\ref{prop:cond3}) are fulfilled and the proof is completed.
\end{proof}
\begin{proposition}\label{prop:thm1_3}
	Under assumptions of Theorem~\ref{thm:MSPE_dec_case1}, it holds that:
\begin{equation}\label{eqn:prop3}
  (\mathrm{III})   = - (D) + o\left(1\right),
\end{equation}
where
\begin{align*}
	(D) &= \E\left[ R^{-1} \left[\sum_{j=h}^{N-1} \left\{\left( \beps_1 x_1\right) \left(\beps_{j+1} x_{j+1}\right)^\top + \left(\beps_{j+1} x_{j+1}\right) \left( \beps_1 x_1\right)^\top\right\} \right] \right]
\end{align*}
\end{proposition}
\begin{proof}
  The result readily follows upon noting that, under Assumption (C2) and the weakly stationarity of the process $\{x_t\}$, it holds that:
  \begin{align*}
(\mathrm{III})  & = - \sum_{t=1}^{N}  \E\left[R^{-1}  \left\{ \left(\beps_t x_t\right) \left(\beps_n x_n\right)^\top +  \left(\beps_n x_n\right) \left(\beps_t x_t\right)^\top\right\} \right]\\
&	= - \ \sum_{j=h}^{n-1}  \E\left[R^{-1}  \left\{ \left(\beps_1 x_1\right) \left(\beps_{j+1} x_{j+1}\right)^\top +  \left(\beps_{j+1} x_{j+1}\right) \left(\beps_1 x_1\right)^\top\right\} \right]\\
&= - \E\left[ R^{-1} \left(\sum_{j=h}^{N-1} \left\{\left( \beps_1 x_1\right) \left(\beps_{j+1} x_{j+1}\right)^\top + \left(\beps_{j+1} x_{j+1}\right) \left( \beps_1 x_1\right)^\top\right\} \right) \right] +o(1).
\end{align*}
\end{proof}
\begin{proposition}\label{prop:thm1_4}
	Under assumptions of Theorem~\ref{thm:MSPE_dec_case1}, it holds that:
\begin{equation}\label{eqn:prop4}
(\mathrm{IV})  = (1) + \left(Q\right) + \left(D\right) + o(1),
\end{equation}
where
\begin{align*}
(1)  & = N^{-1} \E\left[ R^{-1} \left\{\sum_{t=1}^{N} \left(\beps_t x_t\right)\left(\beps_t x_t\right)^\top \right\} \right], \\
\left(Q\right) &= \E\left[ R^{-1} \left[\sum_{s=1}^{h-1} \left\{\left( \beps_1 x_1\right) \left(\beps_{s+1} x_{s+1}\right)^\top + \left(\beps_{s+1} x_{s+1}\right) \left( \beps_1 x_1 \right)^\top\right\} \right] \right]\\
	\left(D\right) &= \E\left[ R^{-1} \left[\sum_{j=h}^{N-1} \left\{\left( \beps_1 x_1\right) \left(\beps_{j+1} x_{j+1}\right)^\top + \left(\beps_{j+1} x_{j+1}\right) \left( \beps_1 x_1\right)^\top\right\} \right] \right]
\end{align*}
\end{proposition}
\begin{proof}
  Let
  \begin{equation*}
    (2)  = N^{-1} \E\left[ R^{-1} \left\{\sum_{j=1}^{N-1} \sum_{k=j+1}^{N} \left(\beps_j x_j\right) \left(\beps_k x_k\right)^\top\right\} \right],
  \end{equation*}
  and note that $(\mathrm{IV})-(1) =  (2) + (2)^\top$. Moreover
  \begin{align}
  	(2) & = N^{-1} \E\left[ R^{-1} \left\{\sum_{j=1}^{N-1} \left( N-j \right) \left( \beps_1 x_1\right) \left(\beps_{j+1} x_{j+1}\right)^\top \right\} \right] \nonumber\\
	& = \E\left[ R^{-1} \left\{\sum_{j=1}^{N-1} \left( \beps_1 x_1\right) \left(\beps_{j+1} x_{j+1}\right)^\top \right\} \right] \label{eqn:a} \\
	& - N^{-1} \E\left[ R^{-1} \left\{\sum_{j=1}^{N-1} j\left( \beps_1 x_1\right) \left(\beps_{j+1} x_{j+1}\right)^\top \right\} \right].\label{eqn:b}
\end{align}
Assumptions (C2) implies that (\ref{eqn:b}) is $o(1)$. Since (\ref{eqn:a}) can be written as
\begin{equation*}
  \E\left[ R^{-1} \left\{\sum_{s=1}^{h-1} \left( \beps_1 x_1\right) \left(\beps_{s+1} x_{s+1}\right)^\top \right\} \right] + \E\left[ R^{-1} \left\{\sum_{j=h}^{N-1} \left( \beps_1 x_1\right) \left(\beps_{j+1} x_{j+1}\right)^\top \right\} \right],
\end{equation*}
then $(\ref{eqn:a})+(\ref{eqn:a})^\top=(Q)+(D)$ and this completes the proof.
\end{proof}
\subsubsection*{Proof of Theorem~\ref{thm:MSPE_dec_case1}}
We prove that:
\begin{align}
&N \left\{ \E\left[ \left( \by_{n+h} - \hat{\by}_{n+h} \right) \left( \by_{n+h} - \hat{\by}_{n+h} \right)^\top - \E\left[
		\beps_{n}^{(h)} \beps_{n}^{(h)^\top} \right] \right] \right\}\nonumber\\
		& =  R^{-1} \E\left[   \left(\beps_{1}^{(h)} x_1\right)\left(\beps_{1}^{(h)} x_1\right)^\top \right] \label{eq:main1} \\
		&+ R^{-1} \E\left[  \sum_{s=1}^{h-1} \left\{\left( \beps_1^{(h)} x_1\right) \left(\beps_{s+1}^{(h)} x_{s+1}\right)^\top + \left(\beps_{s+1}^{(h)} x_{s+1}\right) \left( \beps_1^{(h)} x_1 \right)^\top\right\} \right] \label{eq:main2}\\
&+ o(1). \nonumber
\end{align}
Since $$\left( \boldsymbol{\hat{\beta}}-\bBeta\right)   = \hat{R}^{-1} \left( N^{-1} \sum_{t=1}^{N} x_t \boldsymbol{y}_{t+h} \right) - \bBeta =  \hat{R}^{-1} \left(N^{-1} \sum_{t=1}^{N} x_t \beps_t \right),$$
 routine algebra implies that:
\begin{equation}
\E\left[ \left(\by_{n+h} - \mathbf{\hat{y}}_{n+h}\right) \left(\by_{n+h} - \mathbf{\hat{y}}_{n+h}\right)^\top \right]
- \E\left[\beps_n \beps_n^\top\right] =
(\mathrm{I})+(\mathrm{II}).
\end{equation}
By applying Propositions~\ref{prop:thm1_1} -- Propositions~\ref{prop:thm1_4}, we have:
\begin{align*}
N \left\{ \E\left[ \left( \by_{n+h} - \hat{\by}_{n+h} \right) \left( \by_{n+h} - \hat{\by}_{n+h} \right)^\top - \E\left[
		\beps_{n}^{(h)} \beps_{n}^{(h)^\top} \right] \right] \right\}= N(\mathrm{I})+N(\mathrm{II})\\
=(\mathrm{III})+(\mathrm{IV})+o(1)
=(1)+(Q)+o(1).
\end{align*}
The proof is completed upon noting that $(1)=(\ref{eq:main1})$ and $(Q)=(\ref{eq:main2})$.

% *************************************************************************
\subsection{Proof of Theorem~\ref{thm:MOME_case1}}\label{sec:A2}
% *************************************************************************

We start proving that
\begin{equation}\label{mi_hat}
  \hat{\bmi}_h= \bmi_h + O_p (n^{-1/2}).
\end{equation}
Note that
\begin{align*}
\hat{\bmi}_h=N^{-1} \left(\sum_{t=1}^{N} \beps_{t} \beps_{t}^\top\right)- \left(N^{-1}\sum_{t=1}^{N} x_t \beps_{t}\right) \hat{R}^{-1}\left(N^{-1}\sum_{s=1}^{N} x_s \beps_{s}\right)^\top
\end{align*}
hence, it holds that $\hat{\bmi}_h - \bmi_h$ equals
\begin{align}
& N^{-1} \left\{\sum_{t=1}^{N} \left( \beps_{t} \beps_{t}^\top - \E\left[ \beps_{1} \beps_{1}^\top\right] \right) \right\}\label{eq:mi_A}\\
& -\left(N^{-1}\sum_{t=1}^{N} x_t \beps_{t}\right)\hat{R}^{-1} \left(N^{-1}\sum_{t=1}^{N} x_t \beps_{t}\right)^\top.\label{eq:mi_B}
\end{align}
Assumption (A1) implies that $(\ref{eq:mi_A})=O_p(n^{-1/2})$ whereas, by combining Assumptions (A3) and (A4) with the non-singularity of $R$ and H\"older's inequality, it can be shown that $(\ref{eq:mi_B})=O_p(n^{-1})$ and hence the proof of (\ref{mi_hat}) is complete.
\par\noindent
Next, we prove that
\begin{equation*}
\hat{\bvi}_h= \bvi_h + o_p (1).
\end{equation*}
It suffices to show that
\begin{equation}\label{Cs_hat}
\hat{\bC}_{h,s} = \bC_{h,s}+ o_p (1).
\end{equation}
It holds that $\hat{\bC}_{h,s}$ is equal to
\begin{align}
 &\phantom{-}\left( N-s\right)^{-1} \sum_{t=1}^{N-s} \left( x_t \beps_{t}^{\top}\right)^{\top}  \left( x_{t+s} \beps_{t+s}^{\top}\right)\label{Cs_A}\\
& - \left( N-s\right)^{-1} \sum_{t=1}^{N-s} x_t^2 x_{t+s} \left( \hat{\bBeta}_n(h) - \bBeta_h \right)  \beps_{t+s}^{\top}\label{Cs_B}\\
& - \left( N-s\right)^{-1} \sum_{t=1}^{N-s} x_t x_{t+s}^2 \beps_{t} \left( \hat{\bBeta}_n(h) - \bBeta_h \right)^{\top}\label{Cs_C}\\
& + \left( N-s\right)^{-1} \sum_{t=1}^{N-s} x_t^2 x_{t+s}^2 \left( \hat{\bBeta}_n(h) - \bBeta_h \right)\left( \hat{\bBeta}_n(h) - \bBeta_h \right)^{\top}.\label{Cs_D}
\end{align}
We prove that (\ref{Cs_B}) is $o_p(1)$ componentwise. To this end consider:
$$\E\left[\left( N-s\right)^{-1}\abs*{\sum_{t=1}^{N-s} x_t^2 x_{t+s} \eps_{t+s,i}}\right],$$
with $\eps_{t+s,i}$ being the $i$-th component of the vector $\beps_{t+s}$. The triangular inequality and H\"older's inequality imply that:
\begin{align*}
\E\left[\left( N-s\right)^{-1}\abs*{\sum_{t=1}^{N-s} x_t^2 x_{t+s} \eps_{t+s,i}}\right]
&\leq \left( N-s\right)^{-1}\sum_{t=1}^{N-s}\E\left[\abs*{x_t^2 x_{t+s} \eps_{t+s,i}}\right]\\
&\leq  (N-s)^{-1}\sum_{t=1}^{N-s}\left\{\left(\E\left[x_t^4\right]\right)^{1/2}\left( \E\left[\abs*{x_{t+s}\varepsilon_{t+s,i}}^2	\right]\right)^{1/2}\right\}.
\end{align*}
Since $ \hat{\bBeta}_n(h) - \bBeta_h  = \hat{R}^{-1} \left( N^{-1} \sum_{j=1}^{N} x_j \beps_{j,h} \right)$, by combining Assumptions (A3), (A4) and (A5) with Chebyshev's inequality we obtain that (\ref{Cs_B}) is $o_p(1)$. Similarly, we can verify that (\ref{Cs_C}) and (\ref{Cs_D}) are $o_p(1)$. Lastly, Condition (A2) implies that $(\ref{Cs_A}) = \bC_{h,s}+ o_p (1)$, hence (\ref{Cs_hat}) is verified and the whole proof is complete.

\subsection{Proof of Theorem~\ref{thm:EFF_case1}}\label{sec:A3}

By Theorem \ref{thm:MOME_case1}  the $\vmric_h$ defined in (\ref{eq:vmric})  can be written as:
\begin{equation}\label{eq:theo3_proof_1}
\vmric_h \left( \hat{\ell}_h\right) = \min_{1 \leq \ell \leq K} \norm{ \bmi_h + O_p (n^{-1/2}) } + \min_{ \ell \in M_1} \norm{\frac{\alpha_n}{n}\bvi_h + o_p \left(\frac{\alpha_n}{n}\right)}.
\end{equation}
Therefore,
\begin{equation}
	\lim\limits_{n\rightarrow \infty} \vmric_h \left( \hat{\ell}_h\right) = \min_{1 \leq \ell \leq K}
	\norm{\bmi_h}
\end{equation}
and hence
\begin{equation}
	\lim\limits_{n\rightarrow +\infty} \Pr \left( \hat{\ell}_h \in M_1 \right) = 1.
\end{equation}
Now, consider two models $\ell_1$ and $\ell_2$ in the candidates set $J_{\ell_1}, J_{\ell_2} \in M_1$ such that $\vi_h(\ell_1)\neq \vi_h(\ell_2)$. We show that
\begin{equation}\label{eqn:eff_main}
	\lim_{n\rightarrow \infty}\Pr \left[	\sign{\vmric_h(\ell_1) - \vmric_h(\ell_2)} =	\sign{\norm{\boldsymbol{\vi}_h(\ell_1)}- \norm{\boldsymbol{\vi}_h(\ell_2)}}\right] = 1.
\end{equation}
By defining $\bmi^*_h$ to be the minimum value of $\bmi_h$ over the family of candidate models, we have:
\begin{align*}
	\vmric_h (\ell_1) &= \norm{ \bmi_h^* + O_p (n^{-1/2}) } + \norm{ \frac{\alpha_n}{n}\bvi_h(\ell_1) + o_p\left(\frac{\alpha_n}{n}\right)},\\
	\vmric_h (\ell_2) &= \norm{ \bmi_h^* + O_p (n^{-1/2}) } +	\norm{ \frac{\alpha_n}{n} \bvi_h(\ell_2) + o_p\left(\frac{\alpha_n}{n}\right)}.
\end{align*}
Therefore, for sufficiently large $n$, it holds that:
\begin{equation*}
	\vmric_h (\ell_1) - \vmric_h (\ell_2)	= \left\|\frac{\alpha_n}{n}\right\| \left( \norm{ \bvi_h(\ell_1)} - \norm{ \bvi_h(\ell_2)}\right).
\end{equation*}
 Thus
\begin{equation*}
\sign{ \vmric_h(\ell_1) - \vmric_h(\ell_2)} = \sign{	\norm{\bvi_h(\ell_1)}	 - \norm{\bvi_h(\ell_2)}},
\end{equation*}
and (\ref{eqn:eff_main}) is verified and implies that
\begin{equation}
	\lim_{n\rightarrow \infty}	\Pr\left(\hat{\ell}_h \in M_2\right) = 1.
\end{equation}
This completes the proof.

\section*{Acknowledgments}

Greta Goracci acknowledges the support of Libera Università di Bolzano, Grant WW201L (ESAMD).
\appendix

\section{Technical Lemma}\label{appendx}
The following lemma is a general result that holds for the multivariate predictor case (i.e., $m\geq 1)$. It relies upon Assumption~(C1) of \cite[][p. 1068]{HSU2019}, which reduces to our Assumption~(C1) when $m=1$.
\begin{lemma}\label{lma:technical}
Let $\bR$ and $\hat{\bR}$ be defined in (\ref{eqn:def_R}) and  $\tilde{\bR}$ be the multivariate version of Eq.~(\ref{eqn:R_tilde}). Then, for $0<\gamma\leq 5$, it holds that:
	\begin{align}
\label{eq:lemma1_1}	\E \left[ \norm{\tilde{\bR}^{-1} - \hat{\bR}^{-1} }^\gamma \right] &= O \left[ \left( \frac{l_n}{n} \right)^\gamma \right],\\
\label{eq:lemma1_2}	\E \left[ \norm{\bR^{-1} -  \hat{\bR}^{-1} }^\gamma \right] &= O \left[ n^{-\gamma/2} \right],\\
\label{eq:lemma1_3}	\E \left[ \norm{\bR^{-1} -  \tilde{\bR}^{-1}   }^\gamma \right] &= O \left[ n^{-\gamma/2} \right],
	\end{align}
with $l_n$ being defined in (\ref{eqn:ln}).
\end{lemma}

\begin{proof}
Triangle inequality implies that
\begin{align*}
\norm{\tilde{\bR}^{-1} - \hat{\bR}^{-1} } &\leq \norm{\hat{\bR}^{-1}}\norm{\hat{\bR}-\tilde{\bR}}\norm{\hat{\bR}^{-1}},\\
\norm{\bR^{-1} -  \hat{\bR}^{-1} } &\leq \norm{\hat{\bR}^{-1}}\norm{\hat{\bR}-{\bR}}\norm{{\bR}^{-1}},\\
\norm{\bR^{-1} -  \tilde{\bR}^{-1} } &\leq \norm{\tilde{\bR}^{-1}}\norm{\hat{\bR}-{\bR}}\norm{{\bR}^{-1}}.
\end{align*}
Since $\bR$ is invertible we have that $\norm{{\bR}^{-1}}=O(1)$; under Assumption (C5), $\norm{\hat{\bR}^{-1}}=O(1)$. Moreover, it can be easily proved that also $\norm{\tilde{\bR}^{-1}}=O(1)$. Therefore, by deploying H\"older's inequality, the results will be verified if we prove the following three conditions:
\begin{align}
\E\left[\norm{\hat{\bR}-\tilde{\bR}}\right]&=O\left(\frac{l_n}{n}\right),\label{eqn:lemCond1}\\
\E\left[\norm{\hat{\bR}-{\bR}}\right]&=O\left(n^{-1/2}\right),\label{eqn:lemCond2}\\
\E\left[\norm{\tilde{\bR}-{\bR}}\right]&=O\left(n^{-1/2}\right).\label{eqn:lemCond3}
\end{align}
Let $a=n-l_n$ and $b=a+1$. As for (\ref{eqn:lemCond1}) note that
\begin{align*}
\E\left[\norm{\hat{\bR}-\tilde{\bR}}\right]&\leq \E\left[\norm{\left( \frac{1}{N} - \frac{1}{a}\right)  \sum_{t=1}^{a} \bx_t \bx_t^\top}\right] + \E\left[\norm{ \frac{1}{N}  \sum_{t=b}^{N} \bx_t \bx_t^\top}\right] = O\left(\frac{l_n}{n}\right).
\end{align*}
Conditions (\ref{eqn:lemCond2}) and (\ref{eqn:lemCond3}) readily derive from Assumption (C1) and hence the proof is completed.
\end{proof}

\bibliographystyle{plainnat}
\bibliography{biblio_V5}

\end{document}